\documentclass[12points,reqno]{amsart}
\usepackage{amsfonts}
\usepackage{amssymb}
\usepackage{graphicx}
\usepackage{amsmath}

\newtheorem{theorem}{Theorem}

\theoremstyle{definition}

\newtheorem{example}{Example}

\theoremstyle{remark}

\numberwithin{equation}{section}

\everymath{\displaystyle}

\begin{document}
\title{ On finitely presented algebras }

%

\author{Adel Alahmadi}
\address{Department of Mathematics, Faculty of Science, King Abdulaziz University, P.O.Box 80203, Jeddah, 21589, Saudi Arabia}
\email{analahmadi@kau.edu.sa}
\author{Hamed Alsulami}
\address{Department of Mathematics, Faculty of Science, King Abdulaziz University, P.O.Box 80203, Jeddah, 21589, Saudi Arabia}
\email{hhaalsalmi@kau.edu.sa}
%
%


\maketitle

\begin{abstract}
We prove that if $A$ is finitely presented  algebra with idempotent $e$ such that $A=AeA=A(1-e)A$ then the algebra $eAe$ is finitely presented.
\end{abstract}

\maketitle

\section{Introduction}
Let $F$ be a field and let $A$ be an associative $F$-algebra generated by a finite collection of elements $a_1,\cdots, a_m.$

\medskip

Consider the free associative algebra $F\langle x_1,\cdots, x_m\rangle$ and the homomorphism $F\langle x_1,\cdots, x_m\rangle \rightarrow^{\varphi} A, \, x_i\mapsto a_i.$

\medskip

We say that the algebra $A$ is finitely presented (f.p.) if the ideal $I=\ker\varphi$ is finitely generated as an ideal.
This property does not depend on a choice of a generating system of $A$ as long as this system is finite.

\section{ Main Result}

\begin{theorem}\label{thm1} Let $A=A_{\overline{0}}+A_{\overline{1}}$ be a $\mathbb{Z}/2\mathbb{Z}$- graded associative algebra, such that $A_{\overline{0}}=A_{\overline{1}}A_{\overline{1}}.$
If the algebra $A$ is finitely presented then so is the algebra $A_{\overline{0}}.$

\end{theorem}

\medskip

\par Let $A$  be a finitely generated associative algebra with an idempotent  $e$.

\medskip

S. Montgomery and L. Small [2] showed that if $A=AeA$ then the Peirce component $eAe$ is also finitely generated.

\medskip

\begin{theorem}\label{thm2}
Let $A$ be a finitely presented  algebra with an idempotent $e$ such that $A=AeA=A(1-e)A.$ Then the Peirce component $eAe$ is  finitely presented.

\end{theorem}

\begin{proof}[ Proof of Theorem \ref{thm1}]
By our assumption the algebra $A$ is generated by the subspace $A_{\overline{1}}.$ Let $a_1,\cdots, a_m\in   A_{\overline{1}}$ be the generators of $A.$
Consider the free algebra $F\langle X\rangle, X=\{ x_1,\cdots, x_m\},$ and the homomorphism $f\langle X\rangle\rightarrow^{\varphi} A,\, x_i\mapsto a_i.$
The algebra $F\langle X\rangle$ is $\mathbb{Z}/2\mathbb{Z}$-graded, $F\langle X\rangle = F\langle X\rangle_{\overline{0}}+F\langle X\rangle_{\overline{1}},$
where $F\langle X\rangle_{\overline{0}}$ (resp. $F\langle X\rangle_{\overline{1}}$ ) is spanned by words in $X$ of even (resp. odd ) length. It is easy to see that
the homomorphism $\varphi$ is graded: $\varphi(F\langle X\rangle_{\overline{0}})=A_{\overline{0}},\, \varphi(F\langle X\rangle_{\overline{1}})=A_{\overline{1}}.$
The subalgebra $F\langle X\rangle_{\overline{0}}$ is freely generated by $m^2$ elements $x_ix_j, \, 1\leq i,j \leq m.$
Let $I=\ker\varphi, I=I_{\overline{0}}+I_{\overline{1}},\, I_{\overline{i}}=I\cap F\langle X\rangle_{\overline{i}},\, i=0,1.$
Our aim is to show that $I_{\overline{0}}$ is a finitely generated ideal in $F\langle X\rangle_{\overline{0}}.$

Since the algebra $A$ is finitely presented it follows that the ideal $I$ is generated ( as an ideal ) by a finite set $M=M_{\overline{0}}\cup^{.} M_{\overline{1}},$
$M_{\overline{i}}=F\langle X\rangle_{\overline{i}}, i=0,1.$

Consider the finite set $$M'=\{a, x_ib, bx_i, x_i a x_j\mid a\in M_{\overline{0}}, b\in M_{\overline{1}}, 1\leq i,j\leq m\}$$

We claim that $I_{\overline{0}}=F\langle X\rangle_{\overline{0}} M' F\langle X \rangle_{\overline{0}}.$ Indeed,
The equality $I=F\langle X\rangle M F\langle X \rangle$ implies $$ I_{\overline{0}}=F\langle X\rangle_{\overline{0}} M_{\overline{0}} F\langle X \rangle_{\overline{0}}+F\langle X\rangle_{\overline{1}} M_{\overline{0}} F\langle X \rangle_{\overline{1}}+F\langle X\rangle_{\overline{0}} M_{\overline{1}} F\langle X \rangle_{\overline{1}}+F\langle X\rangle_{\overline{1}} M_{\overline{1}} F\langle X \rangle_{\overline{0}}.$$
Now, $F\langle X\rangle_{\overline{1}}=F\langle X\rangle_{\overline{0}} X=X F\langle X \rangle_{\overline{0}}.$ Hence,
$$F\langle X\rangle_{\overline{1}} M_{\overline{0}} F\langle X \rangle_{\overline{1}}\subseteq F\langle X\rangle_{\overline{0}} XM_{\overline{0}}X F\langle X \rangle_{\overline{0}}\subseteq F\langle X\rangle_{\overline{0}} M' F\langle X \rangle_{\overline{0}};$$
$$F\langle X\rangle_{\overline{0}} M_{\overline{1}} F\langle X \rangle_{\overline{1}}\subseteq F\langle X\rangle_{\overline{0}} M_{\overline{1}}X F\langle X \rangle_{\overline{0}}\subseteq F\langle X\rangle_{\overline{0}} M'F\langle X \rangle_{\overline{0}};$$
$$F\langle X\rangle_{\overline{1}} M_{\overline{1}} F\langle X \rangle_{\overline{0}}\subseteq F\langle X\rangle_{\overline{0}}X M_{\overline{1}} F\langle X \rangle_{\overline{0}}\subseteq F\langle X\rangle_{\overline{0}} M' F\langle X \rangle_{\overline{0}}.$$

This proves the claim and finishes the proof of Theorem \ref{thm1}
\end{proof}

The reverse assertion  is not true. There exists a finitely generated $\mathbb{Z}/2\mathbb{Z}$- graded associative algebra $A=A_{\overline{0}}+A_{\overline{1}}$ such that $A_{\overline{0}}=A_{\overline{1}}A_{\overline{1}},$ the algebra $A_{\overline{0}}$ is  finitely presented, but the algebra $A$ is not.

\begin{example}
Let $A=\langle x,y \mid x^2=yxy=0, xy^{2i+1}x=0 \text{ for } i\geq 1\rangle,$ the elements $x,y$ are odd. The subspace $A_{\overline{0}}$ (resp. $A_{\overline{1}}$ ) is spanned by the products of even ( resp. odd )
length in $x,y.$ It is easy to see that the algebra $A$ is not finitely presented. The subalgebra $A_{\overline{0}}$ is generated by the elements $a=xy, b=y^2, c=yx$ and presented by the relations
$cb=ca=ba=a^2=c^2=0.$

\end{example}

\begin{proof}[ Proof of Theorem \ref{thm2}]
 Let $A_{\overline{0}} =eAe+(1-e)A(1-e), A_{\overline{1}} =eA(1-e)+(1-e)Ae.$ Then $A=A_{\overline{0}}+A_{\overline{1}}$ is a $\mathbb{Z}/2\mathbb{Z}$-grading of the algebra $A.$\\
By the assumptions of the theorem $$eAe=eA(1-e)Ae\subseteq A_{\overline{1}}A_{\overline{1}}$$ and $$ (1-e)A(1-e)=(1-e)AeA(1-e)\subseteq A_{\overline{1}} A_{\overline{1}}.$$\\
By Theorem 1 the algebra $A_{\overline{0}}= eAe \oplus (1-e)A(1-e)$ is finitely presented. It is well known [1] that a direct sum of two algebras is finitely presented if and only if both summands are
finitely presented.

Hence the algebra $eAe$ is finitely presented.
\end{proof}

\section*{Acknowledgement}
This project was funded by the Deanship of Scientific Research (DSR), King Abdulaziz University, under Grant No.
(27-130-36-HiCi). The authors, therefore, acknowledge technical and financial support of KAU.

\end{document}